\documentclass[12pt,reqno]{amsart}

\usepackage{amsmath, amsthm, amscd, amsfonts, amssymb, graphicx, color}
\usepackage[bookmarksnumbered, colorlinks, plainpages]{hyperref}
\hypersetup{colorlinks=true,linkcolor=red, anchorcolor=green, citecolor=cyan, urlcolor=red, filecolor=magenta, pdftoolbar=true}

\catcode`\@=11
\def\omathop#1#2#3{\let\temp=#1\def\letter{#2}
 \ifcat#3_ \let\next\@@olim\else\let\next\@olim\fi\next#3}
\def\@olim{\letter\text{-}\!\temp}
\def\@@olim_#1{\mathchoice{
 \setbox0=\hbox{$\displaystyle\letter\text{-}\!\temp\!\text{-}\letter$}
 \setbox2=\hbox{$\displaystyle\temp$}
 \setbox4=\hbox{$\scriptstyle#1$}
 \dimen@=\wd4 \advance\dimen@ by -\wd2 \divide\dimen@ by2
 \def\next{\letter\text{-}\!\temp_{\hbox to 0pt{\hss$\scriptstyle#1$\hss}}
 \hskip\dimen@}
 \ifdim\wd2>\wd4 \def\next{\@olim_{#1}}\fi
 \ifdim\wd4>\wd0 \def\next{\mathop{\llap{$\letter$-}\!\temp}\limits_{#1}}\fi
 \next}
 {\@olim_{#1}}{\@olim_{#1}}{\@olim_{#1}}}

\def\olim{\omathop{\lim}{o}}

\newcommand{\be}{\begin{equation}}
\newcommand{\ee}{\end{equation}}

\def\phi{\varphi}

\newcommand{\bi}{\begin{itemize}}
\newcommand{\ei}{\end{itemize}}
\newcommand{\bn}{\begin{enumerate}}
\newcommand{\en}{\end{enumerate}}
\def\R{\mathbb{R}}
\def\Q{\mathbb{Q}}
\def\N{\mathbb{N}}

\newcommand{\hide}[1]{} 

\newtheorem{thm}{Theorem}[section]

\newtheorem{lemma}[thm]{Lemma}

\theoremstyle{definition}

\newtheorem{example}[thm]{Example}

\numberwithin{equation}{section}

\begin{document}

\title{Atomic operators  in vector lattices }

\author{Ralph Chill}

\address{Institut f\"ur Analysis \\
Fakult\"at f\"ur Mathematik\\
TU Dresden, Zellescher Weg 12--14, 01062 Dresden, Germany}

\author{Marat Pliev}

\address{Southern Mathematical Institute of the Russian Academy of Sciences\\
362027 Vladikavkaz, Russia}



\begin{abstract}
In this paper we introduce a new class of operators on vector
lattices. We say that  a linear or nonlinear operator $T$ from a vector lattice $E$ to
a vector lattice $F$ is atomic  if there exists a Boolean
homomorphism $\Phi$ from the Boolean algebra $\mathfrak{B}(E)$ of all
order projections on $E$ to  $\mathfrak{B}(F)$ such that
 $T\pi=\Phi(\pi)T$ for every order projection
$\pi\in\mathfrak{B}(E)$. We show that the set of all atomic
operators defined on a vector lattice $E$ with the principal
projection property and taking values in a Dedekind complete vector
lattice $F$, is a band in the vector lattice of all regular
orthogonally additive operators from $E$ to $F$. We give the  formula
for the order projection onto this band, and we obtain an analytic
representation for atomic operators 
between spaces of measurable  functions.
Finally, we consider the procedure of the extension of an atomic map from a
lateral ideal to the whole space.
\end{abstract}

\thanks{The authors are grateful for support by the Deutsche Forschungsgemeinschaft (grant CH 1285/5-1, Order preserving operators in problems of optimal control and in the theory of partial differential equations) and by the Russian Foundation for Basic Research (grant number 17-51-12064)}

\maketitle


\section{Introduction and preliminaries}

Local operators and, more generally, atomic  operators in classical function
spaces find  numerous applications in control theory, the
theory of dynamical systems and the theory of partial differential equations (see \cite{DrPoSt02,PoSt14,Stp04}). The concept of a local operator was in the context of vector lattices first introduced in \cite{Sh76}. It is an abstract form of the well known property of a nonlinear superposition operator and
can be stated in the following form: the value of the image function
on a certain set depends only  on the values of the preimage
function on the same set. In this article we analyse the notion of an atomic operator in the framework of the theory of vector lattices and orthogonally additive
operators. Today the theory of orthogonally additive operators in
vector lattices  is an active area in functional analysis; see
for instance \cite{AbPli17,AbPli18a,Fl17,Fl19,Ku18,Ku16,OrPlRo16,PlPo16,PlFa17,Pl17,PlWe16,TrVi19}.
Abstract results  of this theory can be applied to the theory of
nonlinear integral operators \cite{MaSL90,SL91}, and there are connections with problems of convex geometry \cite{TrVi18}.

Let us introduce some basic facts concerning vector lattices and orthogonally additive operators. We
assume that the reader is acquainted with the theory of vector lattices
and Boolean algebras. For the standard information we refer to
\cite{AbAl02,AlBu06,Ku00a}. All vector lattices below are assumed  to be
Archimedean.

Let $E$ be a vector lattice. A net $(x_\alpha)_{\alpha \in \Lambda}$
in $E$ \textit{order converges} to an element $x \in E$ (notation
$x_\alpha \stackrel{\rm (o)}{\longrightarrow} x$) if there exists a
net $(u_\alpha)_{\alpha \in \Lambda}$ in $E_{+}$ such that $u_\alpha
\downarrow 0$ and $|x_\alpha - x| \leq u_\alpha$ for all $\alpha\in
\Lambda$ satisfying  $\alpha\geq\alpha_{0}$ for some
$\alpha_{0}\in\Lambda$. Two elements $x$, $y$ of the vector lattice $E$
are \textit{disjoint} (notation $x\bot y$), if
$|x|\wedge|y|=0$. The sum $x+y$ of two disjoint elements $x$ and $y$
is denoted by $x\sqcup y$. The equality
$x=\bigsqcup\limits_{i=1}^nx_i$ means that
$x=\sum\limits_{i=1}^nx_i$ and $x_i \bot x_j$ if $i\neq j$. An
element $y$ of  $E$ is called a \textit{fragment}  of an element $x
\in E$, provided $y \bot (x-y)$. The notation $y \sqsubseteq x$
means that $y$ is a fragment of $x$.  If $E$ is a vector lattice and
$x \in E$ then we denote by $\mathcal{F}_x$ the set of all
fragments of $x$. A positive, linear projection $\pi:E\to E$ is said to
be  an {\it order projection} if $0\leq \pi\leq Id$, where $Id$ is the identity operator on $E$. The set of all order projections on $E$ is denoted by
$\mathfrak{B}(E)$. The set $\mathfrak{B}(E)$ is ordered by $\pi\leq
\rho :\Leftrightarrow \pi\circ\rho=\pi$, and it is a Boolean algebra
with respect to the Boolean operations:
\begin{align*}
\pi\wedge\rho & := \pi\circ\rho ; \\
\pi\vee\rho & := \pi+\rho-\pi\circ\rho ; \\
\overline{\pi} & = Id-\pi.
\end{align*}
An element $x$ of a vector lattice $E$ is called a
\textit{projection element} if the band generated by $x$ is a
projection band, and then we denote by $\pi_{x}$ the order projection onto 
the band generated by $x$. A vector lattice $E$ is said to have the
\textit{principal projection property} if every element of $E$ is a
projection element. For example, every $\sigma$-Dedekind complete
vector lattice has the principal projection property.

A (possibly nonlinear) operator $T:E\rightarrow F$ from a vector lattice $E$ into a real vector space is called \textit{orthogonally additive} if $T(x+y)=T(x)+T(y)$ for every  disjoint elements $x$, $y\in E$. It is clear that if $T$ is orthogonally additive, then $T(0)=0$. The set of all orthogonally additive operators from $E$ into $F$, denoted by $\mathcal{OA} (E,F)$, is a real vector space for the natural linear operations.

An operator $T:E\rightarrow F$ between two vector lattices $E$ and $F$ is said to be:
\begin{itemize}
  \item \textit{positive}, if $Tx \geq 0$ for all $x \in E$;
  \item \textit{order bounded}, if $T$ maps order bounded sets in $E$ to order bounded sets in  $F$.
   \item \textit{laterally-to-order bounded},  if for every
  $x\in E$ the set  $T(\mathcal{F}_{x})$ is  order bounded  in $F$;
\end{itemize}
An orthogonally additive operator $T:E\rightarrow F$ is
\begin{itemize}
  \item \textit{regular},  if  $T=T_{1}-T_{2}$ for two positive, orthogonally additive operators $T_{1}$ and $T_{2}$ from $E$ to $F$.
\end{itemize}

An orthogonally additive, order bounded operator $T:E\rightarrow F$
is called an \textit{abstract Urysohn operator}.  This class of
operators was introduced and studied in 1990 by Maz\'{o}n and
Segura de Le\'{o}n \cite{MaSL90}. We notice that the order boundedness is
a restrictive condition for orthogonally additive operator.
Indeed, every operator $T:\mathbb{R}\to\mathbb{R}$ satisfying $T(0)=0$ is orthogonally additive, but not every operator of this form is order bounded. Consider, for instance, the positive function $T$ defined by
\[
T(x)=\begin{cases} \frac{1}{x^{2}}\,\,\,\text{if $x\neq
0$}\\
0\,\,\,\,\,\text{if $x=0$ }.\\
\end{cases}
\]

The notion of a laterally-to-order bounded operator was introduced in
\cite{PlRa18}. It is obviously weaker than the notion of order bounded operator. An orthogonally additive, laterally-to-order bounded operator
$T:E\rightarrow F$ is also called a \textit{Popov operator}.

We denote by $\mathcal{OA}_{+}(E,F)$ the set of all positive, orthogonally
additive operators from $E$ to $F$ (so that $\mathcal{OA} (E,F)$ becomes an ordered vector space with this cone), by $\mathcal{OA}_{r}(E,F) := \mathcal{OA}_{+}(E,F) - \mathcal{OA}_{+}(E,F)$ the regular, orthogonally additive operators, and by $\mathcal{P}(E,F)$ the laterally-to-order bounded, orthogonally additive operators from $E$ to $F$. Also $\mathcal{OA}_{r}(E,F)$ and $\mathcal{P}(E,F)$ are ordered vector spaces. In general, $\mathcal{OA}_{r}(E,F) \not= \mathcal{P}(E,F)$ (see \cite{PlRa18}), but for a Dedekind complete vector lattice $F$ we have the following strong properties of $\mathcal{OA}_{r}(E,F)$ and $\mathcal{P}(E,F)$.

\begin{thm}[{\cite[Theorem~3.6]{PlRa18}}] \label{thm:PK}
Let $E$ and $F$ be vector lattices, and assume that $F$ is Dedekind complete. Then
$\mathcal{OA}_{r}(E,F) = \mathcal{P}(E,F)$, and $\mathcal{OA}_{r}(E,F)$ is a
Dedekind complete vector lattice. Moreover, for every $S$, $T\in
\mathcal{OA}_{r}(E,F)$ and every $x\in E$,
\begin{enumerate}
\item~$(T\vee S)(x)=\sup\{Ty+Sz:\,x=y\sqcup z\}$;
\item~$(T\wedge S)(x)=\inf\{Ty+Sz:\,x=y\sqcup z\}$;
\item~$(T)^{+}(x)=\sup\{Ty:\,y\sqsubseteq x\}$;
\item~$(T)^{-}(x)=-\inf\{Ty:\,\,\,y\sqsubseteq x\}$;
\item~$|Tx|\leq|T|(x)$.
\end{enumerate}
\end{thm}

\section{Basic properties of atomic operators }

In this section we introduce a new subclass of orthogonally additive
operators, namely the class of {\it atomic operators}, and show that under some assumptions on the vector lattices $E$ and $F$ the set
of all atomic operators from $E$ to $F$ subordinate to a Boolean
homomorphism $\Phi:\mathfrak{B}(E)\to\mathfrak{B}(F)$ is a band in
the vector lattice $\mathcal{OA}_{r}(E,F)$. We further obtain a formula
for the order projection onto this band.

Let us first recall the definition of Boolean homomorphisms. 
Let  $\mathfrak{A}$, $\mathfrak{B}$ be Boolean algebras. A map
$\Phi:\mathfrak{A}\to \mathfrak{B}$ is called a \emph{Boolean
homomorphism}, if the following conditions hold:
\begin{enumerate}
  \item[(1)] $\Phi(x\vee y)=\Phi(x)\vee\Phi(y)$ for all $x$, $y\in \mathfrak{A}$.
  \item[(2)] $\Phi(x\wedge y)=\Phi(x)\wedge\Phi(y)$ for all $x$, $y\in \mathfrak{A}$.
  \item[(3)] $\Phi(\overline{x})=\overline{\Phi(x)}$ for all $x\in\mathfrak{A}$.
\end{enumerate}
It is clear that
$\Phi(\mathbf{0}_{\mathfrak{A}})=\mathbf{0}_{\mathfrak{B}}$ and
$\Phi(\mathbf{1}_{\mathfrak{A}})=\mathbf{1}_{\mathfrak{B}}$. If,
moreover,
$\Phi(\bigvee\limits_{\lambda\in\Lambda}x_{\lambda})=\bigvee\limits_{\lambda\in\Lambda}\Phi(x_{\lambda})$
for every family (resp. countable family) $(x_{\lambda})_{\lambda\in\Lambda}$ of  elements
of $\mathfrak{A}$ then $\Phi$ is said to be  an {\it order
continuous} (resp. a {\it sequentially order continuous}) Boolean homomorphism.
 
Let $E$ and $F$ be vector lattices and $\Phi$ be a Boolean homomorphism
from $\mathfrak{B}(E)$ to $\mathfrak{B}(F)$. A map $T:E\to F$ is
said to be an {\it  atomic operator subordinate to $\Phi$}, or briefly {\it
atomic operator}, if $T\pi=\Phi(\pi)T$ for every  order projection
 $\pi\in\mathfrak{B}(E)$. The set of all atomic operators from $E$ to $F$
subordinate to $\Phi$ is denoted by $\Phi(E,F)$.

We remark that the class of atomic operators was first introduced in
\cite{AbPli17}. It is easy to verify that $\Phi (E,F)$ is a vector space. 
Indeed, let $\lambda\in\mathbb{R}$, $T$, $S\in\Phi(E,F)$, $x\in E$ and
$\pi\in\mathfrak{B}(E)$. Then
\begin{align*}
\Phi(\pi)\lambda T(x) & =\lambda\Phi(\pi)T(x)=\lambda T\pi(x) \text{ and} \\
\Phi(\pi)(T+S)(x) & =\Phi(\pi)T(x)+\Phi(\pi)S(x) \\
 & = T\pi(x)+S\pi(x) \\
 & = (T+S)\pi(x) .
\end{align*}

Let us consider some examples of atomic operators.

\begin{example}
Recall that an operator $T:E\to E$ on a vector lattice is said to
be {\it band preserving} if $T(D)\subseteq D$ for every
band $D$ of $E$. By \cite[Theorem 2.37]{AlBu06}, if $E$ is a vector lattice with the principal projection property, then a linear operator $T:E\to E$ is band preserving if
and only if $T$ commutes with every order projection on $E$. In other words, if $E$ is a vector lattice with the principal projection property, then a linear operator $T:E\to E$ is band preserving if and only if it is atomic subordinate to the identity homomorphism  $\Phi:\mathfrak{B}(E)\to\mathfrak{B}(E)$. In particular, if $E$ has the principal projection property, then every linear {\it orthomorphism} $T:E\to E$ (a band preserving, order bounded operator) is atomic with respect to the identity homomorphism.
\end{example}

\begin{example}
Let  $E=l^{p}(\mathbb{Z})$ with $1\leq p\leq\infty$. For every subset $A\subseteq\mathbb{Z}$ one can define an order projection $\pi_A$ which corresponds in fact to the multiplication by the characteristic function $1_A$. This gives a one-to-one correspondence between the Boolean algebra ${\mathcal P} (\mathbb{Z})$ of all subsets of $\mathbb{Z}$ and the Boolean algebra of order projections on $E$. With this identification and for fixed $k\in\mathbb{Z}$, if we define the shift Boolean homomorphism $\Phi_k : {\mathcal P} (\mathbb{Z} ) \to {\mathcal P} (\mathbb{Z})$, $A\mapsto \Phi_k (A) = \{ i+k : i\in A\}$ and the shift operator $T_k : E\to E$, $f\mapsto T_k f$ with $(T_{k}f)(i)=f(i-k)$, then $T_k$ is an atomic operator subordinate to $\Phi_k$.  
\end{example}
 
The following is an example of a nonlinear atomic operator.

\begin{example}\label{Nem}
Let $(B,\Xi,\nu)$ be a $\sigma$-finite  measure space,
$L_{0}(B,\Xi,\nu) = L_{0}(\nu)$ the vector
space of all (equivalence classes of) measurable real valued functions on $B$. A function $N:B\times\mathbb{R}\rightarrow\mathbb{R}$ is called a {\em $\mathfrak{K}$-function} if it satisfies the  conditions:
\begin{enumerate}\label{properties}
 \item[$(C_{0})$] $N(s,0)=0$ for $\nu$-almost all $s\in B$;
 \item[$(C_{1})$] $N(\cdot,r)$ is measurable for all $r\in\mathbb{R}$;
 \item[$(C_{2})$] $N(s,\cdot)$ is continuous on $\mathbb{R}$ for $\nu$-almost all $s\in B$.
\end{enumerate}
If the function satisfies only the conditions $(C_{1})$ and $(C_{2})$, then we call it a \textit{Carath\'{e}odory function}. Given a Caratheodory function $N:B\times\mathbb{R}\rightarrow\mathbb{R}$, one defines
the \textit{superposition operator} $T_{N}:L_{0}(\nu)\to L_{0}(\nu)$ by
\[
T_N f := N(\cdot ,f(\cdot )) \quad (f\in L_{0}(\nu)).
\]
We note that a superposition operator is in the literature also known as {\em Nemytskii operator}. The theory of these  operators is widely represented in the literature (see
\cite{ApZa90}).
\end{example}

\begin{lemma} \label{lem.nemytski}
If $N: B\times\mathbb{R}\rightarrow\mathbb{R}$ is a $\mathfrak{K}$-function, then the superposition operator $T_{N}:L_{0}(\nu)\to L_{0}(\nu)$ is an atomic operator subordinate to the identity homomorphism
$Id:\mathfrak{B}(L_{0}(\nu))\to\mathfrak{B}(L_{0}(\nu))$.
\end{lemma}

\begin{proof}
Let $\Xi_0=\{D\in\Xi:\,\nu(D)=0\}$ be the set of all $\nu$-null sets. Then $\Xi_{0}$ is an ideal in the Boolean algebra
$\Xi$. We let $\Xi':=\Xi/\Xi_0$ be the factor algebra. It is well known that the Boolean algebra $\mathfrak{B}(L_{0}(\nu))$
of all order projections on $L_{0}(\nu)$ is isomorphic to the
Boolean algebra $\Xi'$. In fact, to every (equivalence class of a) measurable subset $D\in\Xi'$ there corresponds an order projection $\pi_D$ which is in fact the multiplication by the characteristic function $1_D$, and vice versa; see \cite[Section 1.6]{AbAl02}. Now  we show that
\[  
N(s,r1_D(s))=N(s,r)1_D(s) \quad \mbox{for all} \; D\in \Xi'.   
\]
First, for every $s\in D$,
\[  
N(s, r1_D(s)) = N(s, r) = N(s,r)1_D(s).   
\]
Second, for every $s\in B\setminus D$, by condition $(C_{0})$,
\[  
N(s, r1_D(s)) = N(s, 0)= 0 = N(s,r)1_D(s).   
\]
Hence, for every $f\in L_{0}(\nu)$ and $\pi = \pi_D\in\mathfrak{B}(L_{0}(\nu))$,
\begin{gather*}
T\pi f = T(f1_{D})=N(\cdot,f1_{D}(\cdot))= N(\cdot,f(\cdot))1_{D}(\cdot)=\pi Tf ,
\end{gather*}
and the assertion is proved.
\end{proof}

The following lemma shows that on Banach lattices with the principal projection property every atomic operator is regular orthogonally additive.

\begin{lemma}\label{Oradd}
Let $E$ be a vector lattice with the principal projection property,
$F$ be a  vector lattice,  $\Phi$ be a Boolean homomorphism from
$\mathfrak{B}(E)$ to $\mathfrak{B}(F)$ and $T\in\Phi(E,F)$. Then $T$
is orthogonally additive, laterally-to-order bounded (that is, $T\in{\mathcal P} (E,F)$) and disjointness preserving. 
\end{lemma}

\begin{proof}
Fix $x$, $y\in E$ with $x\perp y$. Then
\begin{align*}
T(x+y) & =T(\pi_{x}+\pi_{y})(x+y) \\
 & = T(\pi_{x}\vee\pi_{y})(x+y) \\
 & =\Phi(\pi_{x}\vee\pi_{y})T(x+y) \\
 & = (\Phi(\pi_{x})\vee\Phi(\pi_{y})) T(x+y) \\
 & = (\Phi(\pi_{x})+\Phi(\pi_{y}))T(x+y) \\
 & = \Phi(\pi_{x})T(x+y)+\Phi(\pi_{y})T(x+y) \\
 & = T\pi_{x}(x+y)+T\pi_{y}(x+y) \\
 & =Tx+Ty.
\end{align*}
Hence, $T$ is orthogonally additive.

We next show that $T$ is disjointness preserving. Let $x$, $y\in E$ be disjoint elements. 
Then the order projections $\pi_{x}$, $\pi_{y}$ are disjoint
elements in the Boolean algebra $\mathfrak{B}(E)$, and hence $\Phi(\pi_{x})$ and
$\Phi(\pi_{y})$ are disjoint elements in the Boolean algebra
$\mathfrak{B}(F)$. Since
\begin{align*}
Tx & = T\pi_{x}x=\Phi(\pi_{x})Tx \text{ and} \\
Ty & = T\pi_{y}y=\Phi(\pi_{y})Ty ,
\end{align*}
then $Tx\perp Ty$.

Now fix $x\in E$ and assume that $y\in\mathcal{F}_{x}$. Then
by definition $y\perp(x-y)$, and therefore, since $T$ is disjointness preserving, $Ty\perp T(x-y)$. Since $T$ is orthogonally additive, this implies $T(\mathcal{F}_{x})\subseteq\mathcal{F}_{Tx}$, and hence $|Ty|\leq |Tx|$. Hence $T$ is a laterally-to-order bounded. 
\end{proof}

It is worth to notice that without any assumption on the vector 
lattices $E$ and $F$ the space $\mathcal{P}(E,F)$ is
not a vector lattice and we say nothing about the order structure of
the space of laterally-to-order bounded orthogonally additive
operators. Nevertheless the next lemma shows that $\Phi (E,F)$ is a vector lattice
if $E$ has the principal projection property. 

\begin{lemma}\label{module}
Let $E$ be a vector lattice with the principal projection property,
$F$ be a  vector lattice, and $\Phi$ be a Boolean homomorphism from
$\mathfrak{B}(E)$ to $\mathfrak{B}(F)$. Then $\Phi(E,F)$ is a vector lattice. 
\end{lemma}

\begin{proof}
Let $T\in\Phi (E,F)$. It suffices to show that $|T| = T \vee (-T)$ exists. Define the operator $R:E\to F$ by
\begin{gather*}
Rx:=|Tx|,\,x\in E.
\end{gather*}
For every $x\in E$,
\begin{gather*}
Tx\leq|Tx|=Rx \text{ and } (-Tx)\leq|Tx|=Rx.
\end{gather*}
Thus $T\leq R$ and $(-T)\leq R$. Assume that $G$ is an orthogonally
additive operator from $E$ to $F$ such that $T\leq G$ and $(-T)\leq
G$. Then, for every $x\in E$, $Tx\leq Gx$ and $(-T)x\leq Gx$, and therefore $Tx\vee(-Tx)=|Tx|=Rx\leq Gx$. Hence $R=T\vee(-T)$.

We show that $R$ is orthogonally additive. Indeed, let
$x$, $y\in E$ with $x\perp y$. By Lemma~\ref{Oradd}, $T$ is disjointness preserving. Then
\begin{gather*}
R(x+y)=|T(x+y)|=|Tx+Ty|=|Tx|+|Ty|=Rx+Ry.
\end{gather*}
Finally, we show that $R$ is an atomic operator. Let $\pi\in\mathfrak{B}(E)$ and $x\in E$. Then
\begin{gather*}
R\pi(x)=|T\pi(x)|=|\Phi(\pi)Tx|=\Phi(\pi)|Tx|=\Phi(\pi)Rx ,
\end{gather*}
and the proof is complete.
\end{proof}

The following theorem is the first main result of this section.

\begin{thm}\label{thm:1}
Let $E$ be a vector lattice with the principal projection property,
$F$ be a  Dedekind complete vector lattice and $\Phi$ be a Boolean
homomorphism from $\mathfrak{B}(E)$ to $\mathfrak{B}(F)$. Then
$\Phi(E,F)$ is a band in the vector lattice $\mathcal{OA}_{r}(E,F)$
and for any $T$, $S\in\Phi(E,F)$, $x\in E$ the following relations
hold:
\begin{enumerate}
\item~$(T\vee S)x=Tx\vee Sx$;
\item~$(T\wedge S)x=Tx\wedge Sx$;
\item~$(T)^{+}x=(Tx)^+$;
\item~$(T)^{-}(x)=(Tx)^{-}$;
\item~$|T|x=|Tx|$.
\end{enumerate}
\end{thm}

\begin{proof}
By Theorem \ref{thm:PK}, $\mathcal{OA}_r (E,F)$ is a vector lattice, and by Lemmas~\ref{Oradd} and \ref{module} and Theorem~\ref{thm:PK}, $\Phi(E,F)$ is a linear sublattice of $\mathcal{OA}_{r}(E,F) = {\mathcal P} (E,F)$. Suppose $T$, $S\in\Phi(E,F)$ and $x\in E$. By Theorem~\ref{thm:PK},
\begin{gather*}
(T\vee S)(x)=\sup\{Ty+Sz:\, x=y\sqcup z\}\geq Tx\vee Sx.
\end{gather*}
We remark that if $x = y \sqcup z$, then $y=\pi_{y}x=\pi_{y}y$, $z=\pi_{z}x=\pi_{z}z$,
$\Phi(\pi_{y})\perp\Phi(\pi_{z})$, $\Phi(\pi_{x})=\Phi(\pi_{y})+\Phi(\pi_{z})$ and
$\Phi(\pi_{x})(Tx\vee Sx)=Tx\vee Sx$. Hence,
\begin{align*}
Ty+Sz & = T\pi_{y}y+S\pi_{z}z \\
 & = T\pi_{y}x+S\pi_{z}x \\
 & = \Phi(\pi_{y})Tx + \Phi(\pi_{z})Sx \\
 & \leq \Phi(\pi_{y})(Tx\vee Sx)+\Phi(\pi_{z})(Tx\vee Sx) \\
 & = \Phi(\pi_{x})(Tx\vee Sx) \\
 & = Tx\vee Sx.
\end{align*}
Passing to the supremum   in the left-hand side of the above
inequality over all $y$, $z\in\mathcal{F}_{x}$ such that $x=y\sqcup z$
yields to
\begin{gather*}
(T\vee S)(x)\leq Tx\vee Sx
\end{gather*}
and it follows  that  $(T\vee S)(x)= Tx\vee Sx$. Now it is easy to
deduce formulas for the infimum, module, positive and negative parts of
operators.
\begin{align*}
 (T\wedge S)(x) & =-\Big((-T)\vee(-S)(x)\Big)=
-\Big((-Tx)\vee(-Sx)\Big)=Tx\wedge Sx;\\
 T^+(x) & =(T\vee 0)(x)=Tx\vee 0=(Tx)^+;\\
 T^{-}(x) & =(-T\vee 0)(x)=-Tx\vee 0=(Tx)^-;\\
 |T|x & =(T\vee(-T))(x)=Tx\vee(-Tx)=|Tx|.
\end{align*}

Suppose $S\in\mathcal{OA}_{+}(E,F)$, $T\in\Phi(E,F)$, and 
$0\leq S\leq T$. Then $0\leq S\pi(x)\leq
T\pi(x)$ for any $\pi\in\mathfrak{B}(E)$ and $x\in E$. Since
$\Phi(\pi)T=T\pi$ and
$\Phi(\pi)\perp\Phi(\pi^{\perp})=(\Phi(\pi))^{\perp}$, it follows
that
\begin{gather*}
(\Phi(\pi))^{\perp}T\pi(x)=0\Rightarrow(\Phi(\pi)^{\perp})S\pi(x)=(\Phi(\pi))^{\perp}S(x)=0.
\end{gather*}
Thus $S\pi(E)\subseteq\Phi(\pi)S(E)$ and therefore
$\Phi(\pi)S=S\pi$. Hence, $S\in\Phi (E,F)$, and we have shown that $\Phi (E,F)$ is an ideal in $\mathcal{OA}_r (E,F)$. 

Finally we show that $\Phi(E,F)$ is a band in
$\mathcal{OA}_{r}(E,F)$. Assume that $\pi\in\mathfrak{B}(E)$ and
$T_\lambda \stackrel{\rm (o)}{\longrightarrow} T$, where
$(T_\lambda)_{\lambda\in\Lambda}\subseteq\Phi(E,F)$ and
$T\in\mathcal{OA}_{r}(E,F)$. Then we have
\begin{align*}
|T\pi-\Phi(\pi)T| & = |T\pi-T_\lambda\pi+T_\lambda\pi-\Phi(\pi)T| \\
 & \leq |T\pi-T_\lambda\pi|+|\Phi(\pi)T-T_\lambda\pi| \\
 & = |T\pi-T_\lambda\pi|+|\Phi(\pi)T-\Phi(\pi)T_\lambda|.
\end{align*}
Since the net $\Big(|T\pi-T_\lambda\pi|+|\Phi(\pi)T-\Phi(\pi)T_\lambda|\Big)$
order converges to $0$ it follows that $T\pi=\Phi(\pi)T$ for any
$\pi\in\mathfrak{B}(E)$.
\end{proof}

Let $E$ be a vector lattice with the principal projection property
and $F$ be a Dedekind complete vector lattice. Then $\mathcal{OA}_r (E,F)$ is Dedekind complete by Theorem \ref{thm:PK}, and therefore, by a theorem of F. Riesz, every band is a projection band. By Theorem~\ref{thm:1}, every positive
orthogonally additive operator $T:E\to F$ thus has a unique decomposition
$T=T_{1}+T_{2}$ with $0\leq T_{1}\in\Phi(E,F)$ and
$T_{2}\in\Phi(E,F)^{\perp}$. The next theorem, which is the second main result of this section, gives a description of the band projection onto $\Phi (E,F)$.

By $\mathfrak{D}_{0}(E)$, or
$\mathfrak{D}_{0}$ for short, we denote the set of all finite
partitions of the identity operator $Id$, that is
\[
\mathfrak{D}_{0}=\Big\{(\pi_{i}):\,\,\pi_{k}\wedge \pi_{j}=0, k\neq
j;\,\sum_{i=1}^{n}\pi_{i}=Id;\,n\in\mathbb{N}\Big\}.
\]

\begin{thm}\label{th-2}
Let $E$ be a vector lattice with the principal projection property
and $F$ be a Dedekind complete vector lattice and
$T\in\mathcal{OA}_{+}(E,F)$. Then the component $T_{1}\in\Phi(E,F)$
is given by
\begin{gather*}
\inf\Big\{\sum_{i=1}^{n}\Phi(\pi_{i})T\pi_{i}:\,(\pi_{i})\in\mathfrak{D}_{0}\Big\}.
\end{gather*}
\end{thm}

\begin{proof}
For any $T\in\mathcal{OA}_{+}(E,F)$, set
\[
\mathfrak{A}(T) := \Big\{\sum_{i=1}^{n}\Phi(\pi_{i})T\pi_{i}:\,(\pi_{i})\in\mathfrak{D}_{0}\Big\}.
\]
Clearly $\mathfrak{A}(T)$ is a downward directed set of positive
orthogonally additive operators and taking into account the Dedekind
completeness of the vector lattice $\mathcal{OA}_{r}(E,F)$ we deduce
that there exists $R(T):=\inf\mathfrak{A}(T)$. We verify the
following properties for every $T\in\mathcal{OA}_{+}(E,F)$:
\begin{enumerate}
  \item[$(1)$] $0\leq R(T)\leq T$;
  \item[$(2)$] $R:\mathcal{OA}_{r}(E,F)\to \mathcal{OA}_{r}(E,F)$ extends to a linear operator;
  \item[$(3)$] $R(T)=T\Leftrightarrow T\in\Phi(E,F)$;
  \item[$(4)$] $R(R(T))=R(T)$.
\end{enumerate}
The relation $(1)$ is obvious. In order to prove $(2)$, we show that $R$ is additive on the positive cone. If $T_{1}$, $T_{2}\in \mathcal{OA}_+ (E,F)$, then for arbitrary $(\pi_{i})$, $(\pi_{j})\in\mathfrak{D}_{0}$ we have
\begin{align*}
& \sum\limits_{i}\Phi(\pi_{i})T_{1}\pi_{i}+
\sum\limits_{j}\Phi(\pi_{j})T_{2}\pi_{j} \\
& \geq \sum\limits_{k}\Phi(\pi_{k})(T_{1}+T_{2})\pi_{k} \\
& = \sum\limits_{k}\Phi(\pi_{k})T_{1}\pi_{k}+
\sum\limits_{k}\Phi(\pi_{k})T_{2}\pi_{k},
\end{align*}
where  $(\pi_{k})\in\mathfrak{D}_{0}$  is   finer  than $(\pi_{i})$
and $(\pi_{j})$.  Taking  the  infimum,  we obtain
\begin{align*}
R(T_{1})+R(T_{2})=R(T_{1}+T_{2}).
\end{align*}
Let us prove the equivalence in $(3)$. Assume that $0\leq T\in\Phi(E,F)$. We notice that  for every $(\pi_{i})\in\mathfrak{D}_{0}$
we have that $\Phi(\pi_{i})=Id$. Thus
\begin{align*}
\sum\limits_{i}\Phi(\pi_{i})T\pi_{i} = \sum\limits_{i}\Phi(\pi_{i})^2T = \sum\limits_{i}\Phi(\pi_{i})T = T.
\end{align*}
Passing to infimum on the left hand side of the above equality over all $(\pi_{i})\in\mathfrak{D}_{0}$ we get that $R(T)=T$. On the other hand, assume that $R(T)=T$. We show that $T\in\Phi(E,F)$.  Indeed, fix $\pi\in\mathfrak{B}(E)$.
Then $T\leq \Phi(\pi)T\pi+\Phi(Id-\pi)T(Id-\pi)$. It follows that $\Phi(\pi)T\leq\Phi(\pi)T\pi$ and therefore
$\Phi(\pi)T=T\pi$.

It  remains  to  verify the equality $(4)$. Suppose that $W=R(T)$,
with $T\in\mathcal{OA}_{+}(E,F)$. For  every
$\rho\in\mathfrak{B}(E)$, we may write
\begin{align*}
W\rho & =\inf\Big\{\sum\limits_{i}\Phi(\pi_{i})T\pi_{i}\rho:
\,(\pi_{i})\in\mathfrak{D}_{0} \Big\} \\
 & = \inf\Big\{\sum\limits_{i}\Phi(\pi'_{i})T\pi'_{i}\rho:
\,\sum\limits_{i} \pi'_{i} = \rho \Big\} \\
 & = \inf\Big\{\sum\limits_{i}\Phi(\rho)\Phi(\pi'_{i})T\pi'_{i}:
\,\sum\limits_{i} \pi'_{i} = \rho \Big\} .
\end{align*}
Thus,  $W\rho=\Phi(\rho)W$  for every $\rho\in\mathfrak{B}(E)$. By
the equivalence $(3)$  which is established above,  we obtain
$W=R(W)$.
\end{proof}

We remark that a similar theorem for  orthomorphisms was proved in \cite{Sc80}.

\section{Atomic operators  in  spaces  of measurable functions}
\label{sec3}

In this section  we investigate atomic operators in  spaces of real valued, measurable functions and get an analytic representation for this class of operators.\\

Let $(B,\Xi,\nu)$ be a $\sigma$-finite measure space. Choose an {\em equivalent} finite measure $\lambda$ on $\Xi$ such that $\nu$ and $\lambda$ have the same sets of measure $0$. As in Example~\ref{Nem} above, we denote by $L_{0}(B,\Xi,\nu)$ (or $L_{0}(\nu)$ for brevity) the set of all real valued, measurable functions on $B$. More
precisely, $L_0(\nu)$ consists of equivalence
classes of such functions, where as usual two functions $f$ and $g$ are said to be  equivalent if they coincide almost everywhere on $B$; note that $L_0 (\nu)$ and $L_0 (\lambda )$ coincide.
The vector space $L_0(\nu)$ with the
metric $\rho_{L_0}$, defined by
\[
\rho_{L_0}(f,g) := \int_{B}\frac{|f(s)-g(s)|}{1+|f(s)-g(s)|}\,d\lambda \quad (f,\, g\in L_0 (\nu )) ,
\]
becomes a complete metric space, and the convergence with respect to the metric $\rho_{L_0}$ is equivalent to the convergence in measure, meaning here the measure $\lambda$. Recall that $(f_n)$ converges to $f$ in measure (notation $f_n\overset{\lambda}\rightarrow f$; see \cite[Theorem~1.82]{AbAl02}), if, for every $\delta >0$, $\lim_{n\to\infty} \lambda (\{ |f_n-f|>\delta\}) = 0$. More precisely, the convergence in measure is characterised by the following statement.

\begin{lemma}[{\cite[Theorem~1.82]{AbAl02}}] \label{meas-1}
Let $(B,\Xi,\nu)$ be a $\sigma$-finite measure space. Choose an equivalent finite measure $\lambda$ on $\Xi$ such that $\nu$ and $\lambda$ have the same sets of measure $0$.
For a sequence $(f_{n})\subseteq L_{0}(\nu )$ and and element $f\in L_0 (\nu )$ the following equivalent:
\begin{enumerate}
\item~$f_n\overset{\lambda}\rightarrow f$;\\
\item~every subsequence of $(f_n)$ has a subsequence that converges pointwise $\nu$-almost everywhere to $f$;\\
\item~for every $D\in\Xi$ with $\nu (D)<\infty$,
\[
\lim\limits_{n\to\infty}\int_{D}\frac{|f_n(s)-f(s)|}{1+|f_n(s)-f(s)|}\,d\nu=0.
\]
\end{enumerate}
\end{lemma}

We say that a function
$N:B\times\mathbb{R}\rightarrow\mathbb{R}$ is a \textit{superpositionally measurable function}, or briefly that it is {\em sup-measurable}, if
\begin{enumerate}
 \item[$(C'_{1})$] $N(\cdot,f(\cdot))$ is measurable for every $f\in L_{0}(\nu)$.
\end{enumerate}
We call the function $N$ an {\em  $\mathfrak{S}$-function}, if it is sup-measurable and if 
\begin{enumerate} 
\item[$(C_{0})$] $N(s,0)=0$ for $\nu$-almost all $s\in B$.
\end{enumerate}
This normalisation condition already appeared in Example~\ref{Nem}, where we also defined $\mathfrak{K}$-functions and Caratheodory functions. Note that every sup-measurable function $N$ satisfies automatically the condition $(C_1)$ from Example~\ref{Nem}, that is, $N(\cdot,r)$ is measurable for all $r\in\mathbb{R}$. Indeed, it suffices to identify $r\in\mathbb{R}$ with the corresponding constant function $r1_{B}$. It is well known that every Carath\'{e}odory function  $N$ is sup-measurable; see for instance \cite[Chapter 1.4]{ApZa90}. Sup-measurability of a function $N:B\times\mathbb{R}\rightarrow\mathbb{R}$ is the weakest condition under which the superposition operator $T_N$ given by 
\[
 T_N f :=  N(\cdot ,f(\cdot )) \quad (f\in L_0 (\nu )) 
\]
is well-defined on $L_0 (\nu )$.

Given two sup-measurable functions $N$, $K : B\times \mathbb{R}\to\mathbb{R}$, we write $N\preceq K$ if, for every $f\in L_{0}(\nu)$, $N(\cdot ,f(\cdot ))\leq K(\cdot ,f(\cdot ))$
$\nu$-almost everywhere on $B$. We say that $N$ and $K$ are {\em sup-equivalent} (notation $N\simeq K$) if both $N\preceq K$ and $K\preceq N$.

Let $N:B\times\mathbb{R}\rightarrow\mathbb{R}$ be an $\mathfrak{K}$-function, and let $T_N$ be the associated superposition operator on $L_0 (\nu )$. Then,  
by Example~\ref{Nem}, $T_N$ is atomic with respect to the identity Boolean homomorphism. In addition, since $L_0 (\nu )$ is Dedekind complete and by Lemma \ref{Oradd} and Theorem \ref{thm:PK}, $T_N$ is regular orthogonally additive and disjointness preserving, but actually these two properties can easily be verified directly for a superposition operator. The main result of this section shows that in $L_0 (\nu )$, and {\em up to Boolean homomorphisms}, all sequentially order continuous, atomic operators are superposition operators. Before stating the precise statement, let us recall a definition and introduce an operator. 

We recall that an orthogonally additive operator $T:E\to F$ is  {\it
sequentially order continuous}, if for every  order convergent sequence
$(x_{n})_{n}\subseteq E$ with $x_{n}\overset{\rm
(o)}\longrightarrow x$ the sequence $(Tx_{n})_{n}\subseteq F$ is order convergent to $Tx$.

Now, let $(A,\Sigma,\mu)$ be a second $\sigma$-finite measure space. Recall that for every measurable set $A'\in\Sigma$ the multiplication operator $\pi_{A'}$ associated to the multiplication by the characteristic function $1_{A'}$ is an order projection on $L_0 (\mu )$. In fact, every order projection is of this form, and when we consider the factor Boolean algebra $\Sigma' =\Sigma / \Sigma_0$ as in Example~\ref{Nem} (factorization by the sets of $\mu$-measure zero), then we obtain a one-to-one correspondence, that is, the Boolean algebras $\Sigma'$ and $\mathfrak{B} (L_0 (\mu ))$ are isomorphic. 

Now let $\Phi: \mathfrak{B} (L_0 (\mu )) \to \mathfrak{B} (L_0 (\nu ))$ be a Boolean homomorphism. We identify it with a Boolean homomorphism $\Phi : \Sigma'\to\Xi'$, and we define an associated, linear {\em shift operator} $S_{\Phi}:L_{0}(\mu)\to L_{0}(\nu)$ in the
following way. First, for every simple function
$f=\sum\limits_{i=1}^{n}r_{i}1_{A_{i}}\geq 0$ (where $r_{i} \in\mathbb{R}$ and the  $A_{i}\in\Sigma$ are mutually
disjoint) we set
\[
S_{\Phi}f := \sum\limits_{i=1}^{n}r_{i}1_{\Phi(A_{i})} .
\]
The function $S_{\Phi}f$ is a simple function and therefore measurable, and its definition does not depend on the representation of $f$. Note in this context that since $\Phi$ is a Boolean homomorphism, then the $\Phi (A_i)$ are mutually disjoint, too. Second, for every positive, measurable function $f\in L_0 (\mu )^+$ there exists an increasing sequence $(f_{n})$ of positive, simple functions such that $f=\sup\limits_{n}f_{n}$. One can easily show that the sequence $(Sf_n)$ is order bounded in $L_0 (\nu )$. We then put
\[
S_{\Phi} f:= \sup_{n}S_{\Phi} f_{n} \in L_0 (\nu )^+.
\]
This definition of $S_\Phi f$ does not depend on the choice of the approximating sequence $(f_n)$. Finally, for arbitrary $f\in L_0 (\mu )$ we set
\[
S_{\Phi}f := S_{\Phi} f_{+} - S_{\Phi} f_{-} .
\]
The operator $S_\Phi$ thus defined is a linear, positive operator from $L_0 (\mu )$ into  $L_0 (\nu )$.

Now let $\Phi$ be, in addition, a Boolean isomorphism. Then $\Phi^{-1}:\Xi'\to\Sigma'$ is a Boolean isomorphism, too, and it follows from the definition that $S_\Phi$ is invertible and  $S_\Phi^{-1}=S_{\Phi^{-1}}$. In particular, $S_\Phi^{-1}$ is linear and positive, too. We show that $S_\Phi$
is a sequentially order continuous operator. Assume, on the contrary, that $S_\Phi$ is not sequentially order continuous. Then there exists a sequence $(f_n)$ in $L_0(\mu)$ such that $f_n\downarrow 0$ and $S_\Phi f_n\not\downarrow 0$. Passing to an
appropriate subsequence, we can find  $g>0$ such that $S_\Phi f_n \geq g$ for every $n\in\mathbb{N}$. Applying $S_\Phi^{-1}$ to this inequality yields $f_n \geq S_\Phi^{-1} g$ for every $n\in\mathbb{N}$. Since $S_\Phi^{-1} g >0$, this yields a contradiction.

The next theorem is the main result of the section.

\begin{thm}\label{Rep}
Let  $(A,\Sigma,\mu)$ and $(B,\Xi,\nu)$ be $\sigma$-finite measure
spaces, $\Phi:\Sigma'\to\Xi'$ be a Boolean isomorphism, and
$T:L_{0}(\mu)\to L_{0}(\nu)$ be a regular orthogonally additive
operator. Then the following statements are equivalent:
\begin{enumerate}
\item~$T$ is a continuous (with respect to the metric $\varrho_{L_0}$,  atomic operator subordinate to $\Phi$;
\item~there exists a $\mathfrak{K}$-function
$N :B\times\mathbb{R}\rightarrow\mathbb{R}$ such that $T=T_N\circ S_\Phi$, where $T_N$ is the superposition operator associated with $N$ and $S_\Phi$ is the shift operator associated with $\Phi$, that is,
    \begin{align}\label{represent}
    Tf = N (\cdot ,S_{\Phi}f(\cdot )) \quad (f\in L_0 (\mu )) .
    \end{align}
\end{enumerate}
\end{thm}

\begin{proof}
$(1)\Rightarrow(2)$. Let $T :L_0 (\mu ) \to L_0 (\nu )$ be a  continuous, atomic operator subordinate to $\Phi$. Then we define a function $\hat{N} :B\times\mathbb{R}\to \mathbb{R}$ by
\[
\hat{N} (\cdot, r):=T(r1_{A})(\cdot) \quad (r\in\mathbb{R} ) .
\]
We note that $\hat{N} (\cdot,0)=T(0)=0$ and therefore $\hat{N}(\cdot ,0)=0$ $\nu$-almost everywhere. Moreover, $\hat{N} (\cdot,r)$ is $\Xi$-measurable for every $r\in\mathbb{R}$. Now, take a simple function $f=\sum\limits_{i=1}^{n}r_{i}1_{A_{i}}$, where the $A_{i}$ are mutually
disjoint measurable subsets of $A$ and $r_{i}\in\mathbb{R}$, $1\leq
i\leq n$. Then
\allowdisplaybreaks{
\begin{align*}
 Tf & = T\Big(\sum\limits_{i=1}^{n} r_{i}1_{A_{i}}\Big) \\
 & = \sum\limits_{i=1}^{n} T(r_{i}1_{A_{i}}) \\
 & = \sum\limits_{i=1}^{n} T\pi_{A_{i}}(r_{i}1_{A}) \\
 & = \sum\limits_{i=1}^{n} \Phi(\pi_{A_{i}})T(r_{i}1_{A}) \\
 & = \sum\limits_{i=1}^{n} \hat{N} (\cdot ,r_{i})\, 1_{\Phi(A_{i})} \\
 & = \sum\limits_{i=1}^{n} \hat{N} (\cdot ,r_{i}1_{\Phi(A_{i})}) \\
 & = \hat{N} (\cdot ,\sum\limits_{i=1}^{n} r_{i} 1_{\Phi( A_{i})}) \\
 & = \hat{N} (\cdot ,S_{\Phi}\Big(\sum\limits_{i=1}^{n}r_{i}1_{A_{i}}\Big) ) \\
 & = \hat{N} (\cdot ,S_{\Phi}f ) .
\end{align*}}
In other words, when we define the operator $T_{\hat{N}} : L_{0}(\nu)\to L_{0}(\nu)$, $T_{\hat{N}} := T\circ S_\Phi^{-1}$, then $T_{\hat{N}} f = \hat{N} (\cdot , f(\cdot ))$ for every finite step function $f$, so that on the space of finite step functions, the operator $T_{\hat{N}}$ acts like a superposition operator. At the same time, $T_{\hat{N}}$ is defined everywhere on $L_0 (\nu )$ and  it is  continuous with respect to the metric $\varrho_{L_0}$ by assumption on $T$ and by sequential order continuity of $S_\Phi^{-1}$ which implies continuity with respect to $\varrho_{L_0}$.  

It is, however, not clear whether $\hat{N}$ is sup-measurable. If it was sup-measurable, then we could invoke \cite[Theorem 1.4]{ApZa90} in order to show that $\hat{N}$ is sup-equivalent to a Caratheodory function $N$. For the construction of a Caratheodory function associated with $T_{\hat{N}}$, we proceed as in the proof of \cite[Lemma 1.7]{ApZa90}, that is, by regularisation and approximation.

We may for our purposes without loss of generality assume that $(B,\Xi ,\nu )$ is a finite measure space. In fact, if this measure space was only $\sigma$-finite, then we could replace the measure $\nu$ by an equivalent finite measure $\lambda$, as in the definition of the metric $\varrho_{L_0}$.  Define, for every $k\in\N$, the function  $\mathfrak{t}_k : \Xi \times L_0 (\nu ) \to \R$  by
\[
 \mathfrak{t}_k (D,f) := \int_D ((-k) \vee (T_{\hat{N}} f) (x) \wedge k ) \; d\nu (x) \quad (D\in\Xi , f\in L_0 (\nu ) ),
\]
and then for every $k\in\N$ and every $\lambda >0$  the regularized function 
$\mathfrak{t}_{k,\lambda} :\Xi \times L_0 (\nu ) \to\R$ by
\begin{align*}
 & \mathfrak{t}_{k,\lambda} (D,f) := \inf_{g\in L_0} [ \mathfrak{t}_k (D,g) + \lambda \int_D ( |g(x) - f(x) | \wedge k )\; d\mu (x) ] . 
\end{align*}
Then, for every $k\in\N$, $D\in\Xi$, $f\in L_0 (\nu )$, $r$, $\hat{r}\in\R$,
\begin{align}
\label{eq.prop.tk.1} & \mathfrak{t}_{k,\lambda}  (D,f) \leq \mathfrak{t}_{k,\lambda'}  (D,f) \leq \mathfrak{t}_k (D,f) 
 \text{ for every } 0 < \lambda \leq \lambda' , \\
\label{eq.prop.tk.2} & \lim_{\lambda\to\infty} \mathfrak{t}_{k,\lambda} (D,f)  = \mathfrak{t}_k (D,f) , \text{ and} \\
\label{eq.prop.tk.3} & | \mathfrak{t}_{k,\lambda} (D,r) - \mathfrak{t}_{k,\lambda} (D,\hat{r}) | \leq \lambda\, \mu (D) \, |r-\hat{r}| ,
\end{align}
where in the third line we identify a real number $r$ with the corresponding constant function $r\, 1_B$. In order to see that $\mathfrak{t}_{k,\lambda} (D,f)  \leq \mathfrak{t}_k (D,f)$ (see \eqref{eq.prop.tk.1}), it suffices simply to take $g=f$ in the definition of $\mathfrak{t}_{k,\lambda}$. Similarly, from the definition one sees that $\mathfrak{t}_{k,\lambda}$ increasing in $\lambda >0$ (see \eqref{eq.prop.tk.1}). The property \eqref{eq.prop.tk.2} follows from the order continuity of $T_{\hat{N}}$. Finally, in order to prove \eqref{eq.prop.tk.3}, fix $f$, $\hat{f}\in L_0 (\nu )$. By definition, for every $g\in L_0 (\nu )$,
\[
 \mathfrak{t}_{k,\lambda} (D,f) \leq \mathfrak{t}_k (D,g) + \lambda \int_D ( |g(x) - f(x) | \wedge k )\; d\nu (x) .
\]
Moreover, for every $\varepsilon >0$ there exists $g_\varepsilon\in L_0 (\nu )$ such that 
\[
 \mathfrak{t}_{k,\lambda} (D,\hat{f}) \geq \mathfrak{t}_k (D, g_\varepsilon ) + \lambda \int_D ( |g_\varepsilon (x) - \hat{f} (x) | \wedge k )\; d\nu (x) -\varepsilon .
\]
When we substract both inequalities and take $g=g_\varepsilon$ in the first inequality, then we obtain 
\begin{align*}
 \mathfrak{t}_{k,\lambda} (D,f) - \mathfrak{t}_{k,\lambda} (D,\hat{f}) & \leq \lambda \int_{D} (|f(x) - \hat{f} (x)|\wedge (2k) ) \; d\nu (x) + \varepsilon ,
\end{align*}
or, when $f=r$ and $\hat{f} = \hat{r}$ are constant functions,
\[
  \mathfrak{t}_{k,\lambda}  (D,r) - \mathfrak{t}_{k,\lambda}  (D,\hat{r}) \leq \lambda\, \mu (D) \, |r-\hat{r}| + \varepsilon .
\]
Since this inequality holds for arbitrary $\varepsilon >0$, and by changing  the roles of $r$ and $\hat{r}$, one obtains \eqref{eq.prop.tk.3}.  

One easily shows that for every $k\in\N$, $\lambda >0$ and $r\in\R$ the function $\mathfrak{t}_{k,\lambda}  (\cdot ,r)$ is a measure on $(B,\Xi )$. By the inequality \eqref{eq.prop.tk.3}, this measure is absolutely continuous with respect to $\nu$. By the Radon-Nikodym theorem, for every $k\in\N$, $\lambda\in\N$ and $r\in\Q$, there exist densities $N_{k,\lambda} (\cdot ,r)$ such that 
\[
 \mathfrak{t}_{k,\lambda}  (D,r) = \int_D N_{k,\lambda}  (x,r) \; d\nu (x) . 
\]
Since we are only dealing with a countable set of parameters $k$, $\lambda$ and $r$, by the definition of $\mathfrak{t}_{k,\lambda}$, and by the properties \eqref{eq.prop.tk.1} and \eqref{eq.prop.tk.3}, there exists a set $D_0\in\Xi$ of $\nu$-measure zero such that, for every $k\in\N$, $\lambda\in\N$, $r$, $\hat{r}\in\Q$ and $x\in B\setminus D_0$,  
\begin{align}
 \label{eq.prop.n.1} & -k \leq N_{k,\lambda}  (x,r) \leq k , \\
 \label{eq.prop.n.2} & N_{k,\lambda}  (x,r) \leq N_{k,\lambda+1} (x,r)  
 , \text{ and} \\
 \label{eq.prop.n.3} & | N_{k,\lambda}  (x,r) - N_{k,\lambda} (x,\hat{r}) |\leq \lambda \, |r-\hat{r}| .
\end{align}
From the last inequality it follows that, for every $k\in\N$, $\lambda\in\N$ and every $x\in B\setminus D_0$, the function $N_{k,\lambda} (x,\cdot )$ uniquely extends to a Lipschitz continuous function on $\R$, which we still denote by $N_{k,\lambda}  (x,\cdot )$. In particular, the functions $N_{k,\lambda}$ are Caratheodory functions (more precisely, they are $\mathfrak{K}$-functions), and the associated superposition operators are continuous on $L_0 (\nu )$ by Example \ref{Nem}. Set 
\begin{align*}
 N_{k} (x,r) & := \sup_{\lambda\in\N} N_{k,\lambda} (x,r) .  
\end{align*}
As a pointwise supremum of (Lipschitz) continuous functions, for every $x\in B\setminus D_0$, the function $N_{k} (x,\cdot )$ is lower semicontinuous. By \eqref{eq.prop.n.2}, for every $x\in B\setminus D_0$, $r\in\R$,
\begin{align*}
 & N_k  (x,r) = \lim_{\lambda\to\infty} N_{k,\lambda}  (x,r) .  
\end{align*}
By \cite[Theorem 1.1]{ApZa90}, $N_{k}$ is a so-called Shragin function. 

By Lebesgue's dominated convergence theorem, for every $k\in\N$, every $D\in\Xi$ and every $r\in\R$,
\begin{align*}
 \int_D N_k (x,r) \; d\nu (x) & = \lim_{\lambda\to\infty} \int_D N_{k,\lambda}  (x,r) \; d\nu (x) \\
 & = \lim_{\lambda\to\infty} \mathfrak{t}_{k,\lambda} (D,r) \\
 & =  \mathfrak{t}_k (D,r) \\
 & = \int_D ((-k) \vee \hat{N} (x,r) \wedge k) \; d\nu (x) .  
\end{align*}
As a consequence, there exists a set $D_1\in \Xi$ of $\nu$-measure zero, such that $D_1\supseteq D_0$ and, for every $x\in B\setminus D_1$ and every $r\in\Q$,
\begin{equation} \label{eq.nk}
 N_k  (x,r)  =  
 (-k) \vee \hat{N} (x,r) \wedge k .
\end{equation}
In particular, the superposition operator associated with the Shragin function $N_k$ and the superposition operator associated with the function $(-k) \vee \hat{N} \wedge k$ coincide on the space of rational step functions (step functions taking values in $\Q$). The latter operator, however, uniquely extends to a continuous operator on $L_0 (\nu )$. By \cite[Theorem 1.3]{ApZa90}, the function $N_k$ already is a Caratheodory function. It remains now to let $k\to\infty$, and to note that $N_k (x,\cdot )$ is uniquely determined on $[-k,k]\cap \Q$ by \eqref{eq.nk}, independently of $k\in\N$, in order to obtain a Caratheodory function $N :B\times \R\to\R$ such that the associated superposition operator $T_N$ coincides with $T_{\hat{N}}$ on the space of rational step functions. Hence, by continuity, $T_N = T_{\hat{N}}$ everywhere on $L_0 (\nu )$.  Since $T_N 0 = 0$, $N$ is in fact a $\mathfrak{K}$-function, and we have proved one implication.
  
$(2)\Rightarrow(1)$. Assume that there exists a
$\mathfrak{K}$-function $N :B\times\mathbb{R}\rightarrow\mathbb{R}$
 such that  for any $f\in L_0 (\nu )$  
\begin{align*}
    Tf = N (\cdot ,S_{\Phi}f (\cdot )) ,
\end{align*}
that is, $T = T_N \circ S_\Phi$. Since any superposition operator associated with a Caratheodory function is order continuous, and since $S_\Phi$ is order continuous, then $T:L_{0}(\mu)\to L_{0}(\nu)$ is order continuous.

By Lemma \ref{lem.nemytski}, the superposition operator $T_N$ is an atomic operator subordinate to the identity homomorphism. By construction, the shift operator $S_\Phi$ is an atomic operator subordinate to $\Phi$. It follows easily that the composition $T = T_N \circ S_\Phi$ is an atomic operator subordinate to $\Phi$. 
The proof is finished.
\end{proof}

\section{An extension of positive atomic operators and laterally continuous orthogonally additive operators }

In this section we show that any atomic operator is
laterally-to-order continuous. We also prove that an atomic,
orthogonally additive map defined on a lateral ideal can be extended
to an atomic orthogonally additive operator defined on the whole space.
 
Let $E$, $F$ be  vector lattices. A net
$(x_{\alpha})_{\alpha\in\Lambda}\subseteq E$  is said to be  {\it
laterally  convergent to $x\in E$} if $x_\alpha \overset{\rm
(o)}\longrightarrow x$ and $ (x_{\beta}-x_{\gamma})\bot x_{\gamma}$
for all $\beta$, $\gamma\in\Lambda$, $\beta\geq\gamma$. In this case we
write $x_\alpha \overset{\rm lat}\longrightarrow x$. An orthogonally
additive operator $T:E\to F$ is said to be  {\it laterally-to-order
continuous}, if for every  laterally convergent net
$(x_{\alpha})\subseteq E$  with $x_{\alpha}\overset{\rm
lat}\longrightarrow x$ the net $(Tx_{\alpha})$ order converges to $Tx$.

The following lemma is a variant of Lemma \ref{Oradd}. 

\begin{lemma}\label{cont}
Let $E$ be a vector lattice with the principal projection property,
$F$ be a vector lattice, $\Phi:\mathfrak{B}(E)\to\mathfrak{B}(F)$ be
an order continuous homomorphism of Boolean algebras and $T\in\Phi(E,F)$.
Then $T$ is orthogonally additive, laterally-to-order continuous and disjointness preserving.
\end{lemma}

\begin{proof}
Take a laterally convergent net $(x_{\lambda})_{\lambda\in\Lambda}$
with $x_\lambda \overset{\rm lat}\longrightarrow x$. Denote by
$\pi_{\lambda}$, $\rho_{\lambda}$ and $\pi$ the order projections onto the
bands $\{x_{\lambda}\}^{\perp\perp}$,
$\{x-x_{\lambda}\}^{\perp\perp}$ and $\{x\}^{\perp\perp}$,
respectively. Since the elements
$x-x_{\lambda}$ and $x_{\lambda}$ are disjoint for any
$\lambda\in\Lambda$ it follows that
\begin{gather*}
T(x)=T(x-x_{\lambda}+x_{\lambda})=T(x-x_{\lambda})+T(x_{\lambda}).
\end{gather*}
Moreover,  the net $(\pi_{\lambda})$ order converges to $\pi$ and
the net $(\rho_{\lambda})$ order converges to
$\mathbf{0}_{\mathfrak{B}(E)}$ in the Boolean algebra
$\mathfrak{B}(E)$. Taking into account that $\Phi$ is an
order continuous homomorphism of Boolean algebras we deduce that the net
$\Phi(\rho_{\lambda})$ converges to $\mathbf{0}_{\mathfrak{B}(F)}$
in the Boolean algebra $\mathfrak{B}(F)$. Hence,
\begin{align*}
|T(x)-T(x_{\lambda})| & =|T\pi(x)-T\pi_{\lambda}(x)| \\
 & = |T(\pi-\pi_{\lambda})(x)| \\
 & = |T\rho_{\lambda}(x)| \\
 & = |\Phi(\rho_{\lambda})Tx| \overset{\rm
(o)}\longrightarrow 0 ,
\end{align*}
and this completes  the proof.
\end{proof}

A subset $D$ of a vector lattice $E$ is said to be  a {\it lateral
ideal} if the following conditions hold:
\begin{enumerate}
\item~if $x\in D$ and $y\in\mathcal{F}_{x}$, then $y\in D$;
\item~if $x$, $y\in D$ and $x\bot y$, then $x+y\in D$.
\end{enumerate}

\begin{example}\label{adm-1}
Let $E$ be a vector lattice. Then any  order ideal in $E$ is a
lateral ideal.
\end{example}

\begin{example}\label{adm-2}
Let $E$ be a vector lattice and $x\in E$. Then $\mathcal{F}_{x}$  is
a lateral ideal.
\end{example}

\begin{example}\label{adm-3}
Let $E$, $F$ be vector lattices and $T:E\to F$ a positive, orthogonally additive operator.
Then the kernel
\[
\text{ker} (T)=\{y\in E:\,T(y)=0\}
\]
is a lateral ideal.
\end{example}

Let  $E$, $F$ be  vector lattices,
$D$ be a lateral ideal in   $E$. A map $T: D\to F$ is said to be
\begin{itemize}
\item \textit{orthogonally additive}, if   $T(x+y)=Tx+Ty$
for every disjoint elements  $x$, $y\in D$;
\item \textit{positive}, if $Tx\geq 0$ for every $x\in D$;
\item \textit{atomic}, if $T\pi=\Phi(\pi)T$ for every  order projection $\pi\in\mathfrak{B}(E)$ .
\end{itemize}

\begin{thm}[{\cite[Theorem~4.4]{PlRa18}}] \label{extension}
Let $E$, $F$ be vector lattices with $F$ Dedekind complete, $D\subseteq
E$ be a lateral ideal, and $T:D\to F$ be a positive, orthogonally
additive operator. Then the operator $\tilde{T} : E\to F$ defined by
\[
\widetilde{T}_{D}x=\sup\{Ty:\,y\in\mathcal{F}_{x}\cap D\} \quad (x\in E) ,
\]
with the interpretation $\sup \emptyset = 0$, is positive, orthogonally additive and laterally-to-order continuous, that is, $\widetilde{T}_{D}\in\mathcal{P}_{+}(E,F)$. Moreover, $\widetilde{T}_{D} x= Tx$ for every $x\in D$.
\end{thm}

The operator $\widetilde{T}_{D}\in\mathcal{P}_{+}(E,F)$ is called
the {\it minimal extension} of the positive,
orthogonally additive operator $T:D\to F$.

We recall the following auxiliary result.

\begin{lemma}[{\cite[Lemma~2]{PlPo16}}] \label{pr:partorder}
Let $E$ be a vector lattice. Then the relation $\sqsubseteq$ is a partial
order on $E$. Moreover for every $x\in E$ the set $\mathcal{F}_{x}$,
partially ordered by $\sqsubseteq$, is a Boolean algebra with the
least element $0$, maximal element $x$, and the Boolean operations
\begin{align*}
z\cup y & := (z^{+}\vee y^{+})-(z^{-}\vee y^{-}) , \\
z\cap y & := (z^{+}\wedge y^{+})-(z^{-}\wedge y^{-}), \\
\overline{z} & := x-z \quad (y,\, z\in \mathcal{F}_{x}).
\end{align*}
\end{lemma}

The next theorem is the main result of this section. It shows that
the minimal extension of an atomic orthogonally additive map is an
atomic operator as well.

\begin{thm}\label{context}
Let $E$ be vector lattice with the principal projection property and
$F$ be a Dedekind complete vector lattice, $D$ be a lateral ideal in
$E$, and $T:D\to F$ be an atomic, positive, orthogonally additive map.
Then the minimal extension $\widetilde{T}_{D}$ of $T$ is an atomic,
 positive, orthogonally additive operator from $E$ to $F$ as well.
\end{thm}

\begin{proof}
By  Theorem~\ref{extension},  the operator $\widetilde{T}_{D}$ is
well defined and $\widetilde{T}_{D}\in\mathcal{OA}_{+}(E,F)$. We
show that $\widetilde{T}_{D}$ is an atomic operator. Take
an order projection $\pi\in\mathfrak{B}(E)$ and $x\in E$. First we
show that $\mathcal{D}:=\{Ty:\,y\in\mathcal{F}_x\cap D\}$ is an upward
directed set. Indeed, take $y$, $z\in\mathcal{D}$. By
Lemma~\ref{pr:partorder} there exists $u\in\mathcal{F}_{x}$ such
that $u=z\cap y$. Then $y':=y-u\in\mathcal{F}_{x}\cap D$ and
$y'\perp z$. Hence $v:=y'+z\in\mathcal{F}_{x}\cap D$, $y\sqsubseteq
v$ and $z\sqsubseteq v$.  Thus, for every $y$, $z\in\mathcal{F}_{x}\cap D$
there exists $v$. Taking into account  that the relation
$z\sqsubseteq x$ implies that $Tz\leq Tx$ we deduce that for every
$y$, $z\in\mathcal{F}_{x}\cap D$ there exists $v\in\mathcal{F}_{x}\cap
D$ such that $Ty\leq Tv$ and $Tz\leq Tv$. Now,
\begin{align*}
\widetilde{T}_{D}\pi(x) & = \sup\{T\pi(y):\,y\in\mathcal{F}_{x}\cap
D\} \\
 & = \sup\{\Phi(\pi)T(y):\,y\in\mathcal{F}_{x}\cap D\}.
\end{align*}
Taking into account that the set $\{T(y):\,y\in\mathcal{F}_{x}\cap D\}$
is upward directed and that $\Phi(\pi)$ is an order continuous positive
linear operator for any order projection $\pi\in\mathfrak{B}(E)$ we get
\begin{align*}
\widetilde{T}_{D}\pi(x) & = \sup\{\Phi(\pi)T(y):\,\,y\in\mathcal{F}_{x}\cap
D\} \\
 & = \olim\limits_{\lambda}\{\Phi(\pi)T(y_{\lambda}):\,y_{\lambda}\in\mathcal{F}_{x}\cap
D\} \\
 & = \Phi(\pi)\Big(\olim\limits_{\lambda}\{T(y_{\lambda}):\,y_{\lambda}\in\mathcal{F}_{x}\cap D\}\Big) \\
 & = \Phi(\pi)\sup\{T(y):\,y\in\mathcal{F}_{x}\cap D\} \\
 & = \Phi(\pi)\widetilde{T}_{D}(x) ,
\end{align*}
and the proof is finished.
\end{proof}


\providecommand{\bysame}{\leavevmode\hbox to3em{\hrulefill}\thinspace}

\bibliographystyle{acm}

\begin{thebibliography}{10}

\bibitem{AbPli17}
{\sc Abasov, N., and Pliev, M.}
\newblock On extensions of some nonlinear maps in vector lattices.
\newblock {\em J. Math. Anal. Appl. 455}, 1 (2017), 516--527.

\bibitem{AbPli18a}
{\sc Abasov, N., and Pliev, M.}
\newblock Disjointness-preserving orthogonally additive operators in vector
  lattices.
\newblock {\em Banach J. Math. Anal. 12}, 3 (2018), 730--750.

\bibitem{AbAl02}
{\sc Abramovich, Y.~A., and Aliprantis, C.~D.}
\newblock {\em An {I}nvitation to {O}perator {T}heory}, vol.~50 of {\em
  Graduate Studies in Mathematics}.
\newblock American Mathematical Society, Providence, RI, 2002.

\bibitem{AlBu06}
{\sc Aliprantis, C.~D., and Burkinshaw, O.}
\newblock {\em Positive operators}.
\newblock Springer, Dordrecht, 2006.
\newblock Reprint of the 1985 original.

\bibitem{ApZa90}
{\sc Appell, J., and Zabrejko, P.~P.}
\newblock {\em Nonlinear superposition operators}, vol.~95 of {\em Cambridge
  Tracts in Mathematics}.
\newblock Cambridge University Press, Cambridge, 1990.

\bibitem{DrPoSt02}
{\sc Drakhlin, M.~E., Ponosov, A., and Stepanov, E.}
\newblock On some classes of operators determined by the structure of their
  memory.
\newblock {\em Proc. Edinb. Math. Soc. (2) 45}, 2 (2002), 467--490.

\bibitem{Fl17}
{\sc Feldman, W.}
\newblock A characterization of non-linear maps satisfying orthogonality
  properties.
\newblock {\em Positivity 21}, 1 (2017), 85--97.

\bibitem{Fl19}
{\sc Feldman, W.}
\newblock A factorization for orthogonally additive operators on {B}anach
  lattices.
\newblock {\em J. Math. Anal. Appl. 472}, 1 (2019), 238--245.

\bibitem{Ku00a}
{\sc Kusraev, A.~G.}
\newblock {\em Dominated operators}, vol.~519 of {\em Mathematics and its
  Applications}.
\newblock Kluwer Academic Publishers, Dordrecht, 2000.
\newblock Translated from the 1999 Russian original by the author, Translation
  edited and with a foreword by S. Kutateladze.

\bibitem{Ku16}
{\sc Kusraeva, Z.~A.}
\newblock On compact majorization of homogeneous orthogonally additive
  polynomials.
\newblock {\em Sibirsk. Mat. Zh. 57}, 3 (2016), 658--665.

\bibitem{Ku18}
{\sc Kusraeva, Z.~A.}
\newblock Powers of quasi-{B}anach lattices and orthogonally additive
  polynomials.
\newblock {\em J. Math. Anal. Appl. 458}, 1 (2018), 767--780.

\bibitem{MaSL90}
{\sc Maz{\'o}n, J.~M., and Segura~de Le{\'o}n, S.}
\newblock Order bounded orthogonally additive operators.
\newblock {\em Rev. Roumaine Math. Pures Appl. 35}, 4 (1990), 329--353.

\bibitem{OrPlRo16}
{\sc Orlov, V., Pliev, M., and Rode, D.}
\newblock Domination problem for {$AM$}-compact abstract {U}ryson operators.
\newblock {\em Arch. Math. (Basel) 107}, 5 (2016), 543--552.

\bibitem{Pl17}
{\sc Pliev, M.}
\newblock Domination problem for narrow orthogonally additive operators.
\newblock {\em Positivity 21}, 1 (2017), 23--33.

\bibitem{PlRa18}
{\sc Pliev, M., and Ramdane, K.}
\newblock Order unbounded orthogonally additive operators in vector lattices.
\newblock {\em Mediterr. J. Math. 15}, 2 (2018), Art. 55, 20.

\bibitem{PlFa17}
{\sc Pliev, M.~A., and Fan, S.}
\newblock Narrow orthogonally additive operators in lattice-normed spaces.
\newblock {\em Sibirsk. Mat. Zh. 58}, 1 (2017), 174--184.

\bibitem{PlPo16}
{\sc Pliev, M.~A., and Popov, M.~M.}
\newblock On the extension of abstract {U}ryson operators.
\newblock {\em Sibirsk. Mat. Zh. 57}, 3 (2016), 700--708.

\bibitem{PlWe16}
{\sc Pliev, M.~A., and Weber, M.~R.}
\newblock Disjointness and order projections in the vector lattices of abstract
  {U}ryson operators.
\newblock {\em Positivity 20}, 3 (2016), 695--707.

\bibitem{PoSt14}
{\sc Ponosov, A., and Stepanov, E.}
\newblock Atomic operators, random dynamical systems and invariant measures.
\newblock {\em Algebra i Analiz 26}, 4 (2014), 148--194.

\bibitem{Sc80}
{\sc Schep, A.~R.}
\newblock Positive diagonal and triangular operators.
\newblock {\em J. Operator Theory 3}, 2 (1980), 165--178.

\bibitem{SL91}
{\sc Segura~de Le\'{o}n, S.}
\newblock Bukhvalov type characterizations of {U}rysohn operators.
\newblock {\em Studia Math. 99}, 3 (1991), 199--220.

\bibitem{Sh76}
{\sc Shragin, I.~V.}
\newblock Abstract {N}emycki\u{\i} operators are locally defined operators.
\newblock {\em Dokl. Akad. Nauk SSSR 227}, 1 (1976), 47--49.

\bibitem{Stp04}
{\sc Stepanov, E.}
\newblock Representation of atomic operators and extension problems.
\newblock {\em Proc. Edinb. Math. Soc. (2) 47}, 3 (2004), 695--707.

\bibitem{TrVi18}
{\sc Tradacete, P., and Villanueva, I.}
\newblock Continuity and representation of valuations on star bodies.
\newblock {\em Adv. Math. 329\/} (2018), 361--391.

\bibitem{TrVi19}
{\sc Tradacete, P., and Villanueva, I.}
\newblock Valuations on {B}anach lattices.
\newblock {\em Int. Math. Res. Not.\/} (2019),
\newblock to appear.

\end{thebibliography}

  \def\ocirc#1{\ifmmode\setbox0=\hbox{$#1$}\dimen0=\ht0 \advance\dimen0
  by1pt\rlap{\hbox to\wd0{\hss\raise\dimen0
  \hbox{\hskip.2em$\scriptscriptstyle\circ$}\hss}}#1\else {\accent"17 #1}\fi}
  \def\cprime{$'$} \def\cprime{$'$} \def\cprime{$'$}

\end{document}